\documentclass[12pt]{amsart}

\usepackage{etex}

\usepackage[leqno]{amsmath}
\usepackage{amscd,latexsym,amsthm,amsfonts,amssymb,amsxtra,mathrsfs,bm,tikz}
\usepackage{fullpage}
\usepackage[mathscr]{eucal}
\usepackage{pictexwd,dcpic}
\usepackage[normalem]{ulem}
\usepackage{hyperref}
\usepackage{enumerate}
\usepackage[all]{xy}
\usepackage{color}

\renewcommand\theequation{\thesection.\arabic{equation}}

\newtheorem{thm}{Theorem}[section]

\newtheorem{prop}[thm]{Proposition}
\newtheorem {conj}[thm]{Conjecture}

\newtheorem {ques/conj}[thm]{Question/Conjecture}

\newtheorem{defn}[thm]{Definition}
\newtheorem{rmk}[thm]{Remark}

\newcommand{\BC}{{\mathbb {C}}}

\newcommand{\BZ}{{\mathbb {Z}}}

\newcommand{\CE}{{\mathcal {E}}}

\newcommand{\CG}{{\mathcal {G}}}

\newcommand{\GL}{{\mathrm{GL}}}

\newcommand{\GSp}{{\mathrm{GSp}}}

\newcommand{\GSO}{{\mathrm{GSO}}}

\newcommand{\Hom}{{\mathrm{Hom}}}

\newcommand{\PGL}{{\mathrm{PGL}}}

\newcommand{\SL}{{\mathrm{SL}}}

\newcommand{\SO}{{\mathrm{SO}}}

\newcommand{\Sp}{{\mathrm{Sp}}}

\newcommand{\tr}{{\mathrm{tr}}}

\numberwithin{equation}{subsection}
\newcounter{keepeqno}

\begin{document}
	
\title{A conjecture for multiplicities of strongly tempered spherical varieties}

\author{Chen Wan}
\address{Department of Mathematics \& Computer Science\\
Rutgers University – Newark\\
Newark, NJ 07102, USA}
\email{chen.wan@rutgers.edu}

\author{Lei Zhang}
\address{Department of Mathematics\\
National University of Singapore, Singapore}
\email{matzhlei@nus.edu.sg}

\date{}



\begin{abstract}
In this paper, we form a conjecture about the multiplicities of all the strongly tempered spherical varieties without Type N root for tempered representations. This generalizes the epsilon dichotomy conjectures in \cite{GGP} and \cite{WZ}.
\end{abstract}

\maketitle

\section{Introduction}
Let $F$ be a local field of characteristic 0, $G$ be a connected reductive group defined over $F$ and $H$ be a closed connected subgroup of $G$. Assume that $H$ is a spherical subgroup of $G$ (i.e. $H$ admits an open orbit in the flag variety of $G$). We say the spherical pair $(G, H)$ is reductive if $H$ is reductive. 

We say that the spherical pair $(G, H)$ is the Whittaker induction of a reductive spherical pair $(G_0, H_0)$ if there exists a parabolic subgroup $P=MN$ of $G$ and a generic character $\xi$ of $N(F)$ such that $M\simeq G_0$ and $H_0$ is contained in the neutral component of the stabilizer of the character $\xi$ in $M$ under the adjoint action. If this is the case, we say $(G,H)$ is the Whittaker induction of $(G_0,H_0,\xi)$ (if $H$ is already reductive, we can just let $(G_0,H_0,\xi)=(G,H,1)$). In this paper, we will restrict ourselves to the case when $(G, H)$ is the Whittaker induction of a reductive spherical pair $(G_0, H_0,\xi)$. We can extend the character $\xi$ to $H(F)$ by making it trivial on $H_0(F)$. For an irreducible smooth representation $\pi$ of $G(F)$ whose central character is trivial on $Z_{G,H}(F)=Z_G(F)\cap H(F)$, we define the multiplicity
$$m(\pi,\xi)=\dim(\Hom_{H(F)}(\pi,\xi)).$$
To simplify the notation we will use $m(\pi)$ instead of $m(\pi,\xi)$ to denote the multiplicity if the choice of $\xi$ is clear. We say that the representation is $(H,\xi)$-distinguished (or just $H$-distinguished if the choice of $\xi$ is clear) if the multiplicity is nonzero. One of the fundamental problems in the {\it Relative Langlands Program} is to study the multiplicity $m(\pi,\xi)$. In general, one expects the multiplicity $m(\pi,\xi)$ to be finite and to detect some functorial structures of $\pi$. We refer the reader to \cite{SV} for a detailed discussion of these types of problems.

Among all spherical pairs, there is a special category called strongly tempered spherical pairs. More precisely, when $H$ is reductive, we say the pair $(G, H)$ is strongly tempered if all the matrix coefficients of tempered representations of $G(F)$ are integrable on $H(F)/Z_{G, H}(F)$ (here $Z_G$ is the center of $G$ and $Z_{G, H}=Z_G\cap H$). When $H$ is not reductive and if the model $(G, H)$ is the Whittaker induction of a reductive spherical pair $(G_0, H_0,\xi)$, then we say that the pair $(G, H)$ is strongly tempered if and only if $(G_0, H_0)$ is strongly tempered. According to the general conjecture of Sakellaridis and Venkatesh in Conjecture 16.5.1 of \cite{SV}, for a strongly tempered spherical pair $(G, H)$, if we assume that the spherical varieties $X=G/H$ do not have a Type N spherical root (we refer the reader to Section 3.1 of \cite{SV} for the definition of spherical roots), then almost all the tempered local Vogan $L$-packets of $G(F)$ should contain at least one $(H,\xi)$-distinguished representation (i.e. almost all tempered local Vogan $L$-packets are $(H,\xi)$-distinguished). The key point is that in a strongly tempered case, the $L$-group of the spherical variety $X=G/H$ is the $L$-group of $G$, and hence one expects that almost all tempered local Vogan $L$-packet should be distinguished. Moreover, if the spherical variety only has one open Borel orbit over the local field $F$, then the general conjecture of Sakellaridis and Venkatesh predicts that almost all tempered local Vogan $L$-packets of $G(F)$ should contain exactly one $(H,\xi)$-distinguished representation (this is usually called a strong multiplicity one on $L$-packets). In general, we expect the multiplicity of each tempered local Vogan $L$-packet of $G(F)$ to be equal to the number of open Borel orbits of $X(F)$.

The most famous examples of strongly tempered spherical pairs without Type N root are the so-called Gan--Gross--Prasad models $(\SO_{n+2k+1}\times \SO_n,\SO_n\ltimes N)$ and $(U_{n+2k+1}\times U_n, U_n\ltimes N)$. Here $U$ is some unipotent subgroup. For these cases, the local conjecture was formulated by Gan, Gross, and Prasad in Section 17 of \cite{GGP}. In it, they not only conjectured the property of strong multiplicity one on generic $L$-packets (i.e. each generic L-packet contains a unique distinguished element and its multiplicity is equal to one), but they also conjectured about the unique distinguished element in each $L$-packet. More precisely, for each local $L$-packet $\Pi_\phi$  ($\phi:W_F'\rightarrow {}^LG$ is a Langlands parameter), let $Z_\phi$ be the centralizer of the parameter and $S_\phi=Z_\phi/Z_{\phi}^{\circ}$ be its component group. The local Langlands conjecture states that there is a natural bijection between the $L$-packet and the set of irreducible representations of $S_\phi$ (denoted by $\hat{S}_\phi$). In Section 17 of \cite{GGP}, they defined a quadratic character of $S_\phi$ using some local epsilon factor and conjectured that the unique distinguished element in a generic $L$-packet is the one associated with this quadratic character. This is usually called the epsilon dichotomy conjecture. In our previous paper \cite{WZ}, we formulated the epsilon dichotomy conjecture for 10 strongly tempered spherical varieties and we proved the conjecture in many cases including all the archimedean cases.

In this paper, we will make a general epsilon dichotomy conjecture for the Whittaker induction of any strongly tempered spherical varieties without Type N root, and in Section \ref{sec multiplicity one case} we will show that our conjecture recovers the conjectures in \cite{GGP} and \cite{WZ}. The most important advantage of our conjecture is that, unlike the conjectures in \cite{GGP} and \cite{WZ}, our conjecture does not rely on specific knowledge of the component group $S_\phi$ and the centralizer $Z_\phi$ (in \cite{GGP}, the authors wrote down the component group $S_\phi$ explicitly and then define the character on it; in \cite{WZ}, we specifically write down elliptic elements in $Z_\phi$ and then define the function on it explicitly). The reason we can do this is that based on all the existing examples of strongly tempered spherical varieties, we find that the L-function associated with strongly tempered spherical varieties should satisfy a property called anomaly free. During our preparation of this paper, we were very happy to learn that in the work of Ben-Zvi--Sakellaridis--Venkatesh \cite{BSV}, they also find the same property and it also serves as a key ingredient in their proposed relative Langlands duality (the name ``anomaly free" comes from their paper). We refer the reader to Section 2 for more details. One of the key points for anomaly free is that it allows us to take ``square root" of the local epsilon factor (or the global L-function in the setting of \cite{BSV}).

Another advantage of our conjecture is that it applies to a general strongly tempered case, we do not even need to assume that the model has a unique open Borel orbit (in particular, it may not have strongly multiplicity one over the L-packet). In Section \ref{sec non-multiplicity one case}, we will discuss some examples with more than one open Borel orbit, and we will show that our conjecture holds for these models.

\begin{rmk}
In \cite{Pras}, Prasad gave a beautiful conjecture for the multiplicity of Galois model $(G, H)=(Res_{E/F}H, H)$ where $E/F$ is a quadratic extension. The Galois model case and the strongly tempered case are the two extreme cases in terms of the behavior of multiplicity. The Galois model case is purely related to functoriality, while the strongly tempered case is purely related to the epsilon dichotomy. We believe for general spherical variety without Type N root, the behavior of the multiplicity should lie in between these two extreme cases. In other words, it should be a combination of functoriality and epsilon dichotomy. An example would be the Guo-Jacquet model $(\GL_{2n}(F),\GL_n(E))$ for which the multiplicity is related to both the functoriality and certain epsilon factor. We are currently trying to combine these two conjectures to make a conjecture of the multiplicity for general spherical variety without Type N root.
\end{rmk}

The paper is organized as follows. In Section 2 we will discuss the endoscopic datum, the local Langlands conjecture, and the anomaly free representation of L-groups. Then we will state our conjecture. In Section 3 we will show that our conjecture recovers the conjectures in \cite{GGP} and \cite{WZ}. In Section 4, we will prove our conjecture for some cases with more than one open Borel orbit.

\textbf{Acknowledgement}: We thank Raphael Beuzart-Plessis, Tasho Kelatha, Yiannis Sakellaridis, Akshay Venkatesh, and Jun Yu for the helpful discussions.
The work of the first author is partially supported by the NSF grant DMS-2000192 and DMS-2103720. 
The work of the second author is partially supported by AcRF Tier 1 grants 	A-0004274-00-00 and A-0004279-00-00 of the National University of Singapore.

\section{Representation of $L$-group}

\subsection{Extended endoscopic triple}
Let $G$ be a connected reductive group defined over $F$. Following Definition 2 of \cite{K}, we say $(G',s,{}^L\eta)$ is an extended endoscopic triple of $G$ if $G'$ is a quasi-split connected reductive group define over $F$, $s$ is a semisimple elment of $\hat{G}$, and ${}^L\eta$ is an L-embedding from ${}^LG'$ into ${}^LG$ such that the image of ${}^L\eta$ commutes with $s$ and it induces an isomorphism between $\hat{G}'$ and $\hat{G}_s$ (here $\hat{G}_s$ is the neutral component of the centralizer of $s$ in $\hat{G}$).

In this paper, we will restrict ourselves to the case when each endoscopic datum $\CE=(G',\CG',s,{}^L\eta)$ of $G$ (we refer the reader to Definition 1 of \cite{K} for the definition of endoscopic datum) is also an extended endoscopic triple (this is equivalent to say that $\CG'$ in the endoscopic datum is an L-group) so that we only need to consider extended endoscopic triple instead of the more complicated endoscopic datum.

\begin{rmk}
This assumption is true in many cases. For example, when $G$ is a classical group, or when the derived group $G_{der}$ of $G$ is simply connected.
\end{rmk}

\subsection{The local Langlands conjecture}\label{sec L-packets}
In this subsection, we recall the local Langlands conjecture in Conjecture E of \cite{K}. Let $G$ be a quasi-split reductive group defined over $F$ and let $\{G_\alpha \mid \alpha\in H^1(F, G)\}$ be the set of pure inner forms of $G$. Let $\Pi_{irr,temp}(G_{\alpha})$ be the set of irreducible tempered representations of $G_{\alpha}(F)$. The local Langlands conjecture states that 
$$ \bigcup_{\alpha\in H^1(F,G)} \Pi_{irr,temp}(G_\alpha)$$ 
is a disjoint union of finite sets (i.e. the local tempered Vogan $L$-packets)
$$\cup_{\phi} \Pi_{\phi}$$
where $\phi$ runs over all the tempered $L$-parameters of $G$ and $$\Pi_{\phi}=\bigcup_{\alpha\in H^1(F,G)} \Pi_{\phi}(G_\alpha)$$
consists of a finite number of tempered representations with $\Pi_{\phi}(G_\alpha)\subset \Pi_{irr,temp}(G_{\alpha})$ such that the following conditions hold.

\begin{itemize}
\item There is a unique generic element in $\Pi_\phi(G)$ with respect to any Whittaker datum of $G$.
\item For the given Whittaker datum, there is a bijection between $\hat{S_\phi}$, the set of irreducible representations of the component group $S_\phi=Z_\phi/Z_{\phi}^{\circ}$ of the Langlands parameter $\phi$ ($Z_\phi$ is the centralizer of $Im(\phi)$ in $\hat{G}$), and $\Pi_\phi$ (denoted by $\pi\leftrightarrow \chi_\pi$) satisfies the following conditions. 
\begin{itemize}
\item The trivial character of $S_\phi$ corresponds to the unique generic element of $\Pi_\phi(G)$ with respect to the given Whittaker datum.
\item For $\alpha\in H^1(F,G)$, the distribution character $$\theta_{\Pi_{\phi}(G_\alpha)}=\sum_{\pi\in \Pi_\phi(G_{\alpha})}\dim(\chi_\pi)\theta_\pi$$ 
is stable. Moreover,  $\iota(G_{\alpha})\theta_{\Pi_\phi(G_\alpha)}$ is the transfer of $\theta_{\Pi_\phi(G)}$ where $\iota(G_\alpha)$ is the Kottwitz sign.
\item For any $\alpha\in H^1(F,G)$ and $\pi \in \Pi_\phi(G_\alpha)$, the restriction of the central character of $\chi_{\pi}$ to $Z(\hat{G})^{\Gamma_F}$ is equal to $\chi_\alpha$. Here $\chi_\alpha$ is the character of $Z(\hat{G})^{\Gamma_F}$ associated to $\alpha$ via the Kottwitz isomorphism. 
Note that the representation $\chi_\pi$ of the component group can be viewed as a representation of the centralizer $Z_\phi$ of the image of $\phi$, the group $Z(\hat{G})^{\Gamma_F}$ belongs to the center of $Z_\phi$ and hence it makes sense to talk about the restriction of the central character of $\chi_\pi$ to $Z(\hat{G})^{\Gamma_F}$.
\item For $s\in S_\phi$ and for an extended endoscopic triple $(G',s',{}^L\eta)$ of $G$ such that $s'\in sZ_{\phi}^{\circ}$ and $\phi$ factors through ${}^L\eta$, let $\Pi_{\phi,s}(G')$ be the corresponding $L$-packet of $G'$ and let $\theta_{\Pi_{\phi,s}(G')}$ be the distribution character of that packet (which is a stable character on $G'(F)$). Then for $\alpha\in H^1(F,G)$, the character
$$\theta_{\Pi_\phi,\alpha,s}=\sum_{\pi\in \Pi_\phi(G_\alpha)} \tr(\chi_\pi(s))\theta_\pi$$
is the endoscopic transfer of $\iota(G_\alpha)\theta_{\Pi_{\phi,s}(G')}$.
\end{itemize}
\end{itemize}

\subsection{Anomaly free representation of L-groups}
Given a symplectic representation $\rho_X:{}^LG\rightarrow \GL(V)$ of ${}^LG$, for an extended endoscopic triple $(G',s,{}^L\eta)$, let $V_{s,-}$ be the $-1$-eigenspace of $\rho_X(s)$. Then the extended endoscopic triple induces a symplectic representation of ${}^LG'$ on $V_{s,-}$ which will be denoted by $\rho_{X,s,{}^L\eta,-}$.

\begin{defn}\label{defn:anomaly} (see also Definition 5.1.2 and Proposition 5.1.5 of \cite{BSV})
Assume that $G$ is quasi-split. Let $T\subset G$ be the maximal quasi-split torus. We say that a symplectic representation $\rho_X:{}^LG\rightarrow \GL(V)$ of ${}^LG$ is anomaly free if it satisfies the following two conditions.

\begin{itemize}
\item The restriction of the representation $(\rho_X, V)$ to ${}^LT$ can be decomposed into a direct sum of two representations that are dual to each other, i.e.
$$(\rho_X|_{{}^LT},V)\simeq (\rho,W)\oplus (\rho^\vee,W).$$
\item There exists a character $\chi$ of ${}^LT$ and a character $\eta$ of ${}^LG'$ such that $\det(\rho)=\chi^2\cdot \eta|_{{}^LT}$.
\end{itemize}
\end{defn}

\begin{rmk}\label{rmk:anomaly-cases}
\begin{enumerate}
\item If $G$ is split ($\iff$ $T$ is split), the first condition in the definition is always true.
\item The second condition in the definition does not depend on the decomposition $\rho_X=\rho\oplus \rho^\vee$ in the first condition. 
\item \label{item:dist-pol} If $\rho_X=\rho_0\oplus \rho_{0}^\vee$ where $\rho_0$ is a representation of ${}^LG$, then it is anomaly free.
\item If $\rho_X=\rho_1\oplus \rho_2$ with $\rho_i$ being a representation of ${}^LG$ that is anomaly free, then $\rho_X$ is anomaly free.
\end{enumerate}
\end{rmk}

\begin{defn}
We say the symplectic representation $\rho_X$ of ${}^LG$ is anomaly free under endoscopy if for every extended endoscopic triple $(G',s,{}^L\eta)$ of $G$, the symplectic representation $\rho_{X,s,{}^L\eta,-}$ of ${}^LG'$ is anomaly free.
\end{defn}

\begin{rmk}
If $G$ is split adjoint, all its endoscopic groups are split. Then  $\rho_X$ is anomaly free under endoscopy if for any $s\in \hat{G}_{ss}$, the representation of $\hat{G}_s$ on $V_{s,-}$ is anomaly free.
\end{rmk}

\subsection{Multiplicity for strongly tempered spherical varieties}

Let $(G, H)$ be a strongly tempered spherical pair that is the Whittaker induction of $(G_0, H_0,\xi)$. We assume that $G$ has a quasi-split pure inner form and we let $G_{qs}$ be the quasi-split pure inner form of $G$. We also assume that $(Res_{E/F}G_0, Res_{E/F}H_0)$ is strongly tempered for any finite field extension $E$ of $F$ \footnote{this is to avoid those models that are only strongly tempered because $G_0$ is not split. For example, if $G_0$ is compact (say isomorphic to $\SL_1(D)$ for some division algebra $D$ over $F$), then even the model $(G_0, G_0)$ is strongly tempered but it is not strongly tempered after a suitable finite field extension}.

The L-group of the spherical variety $X=H\backslash G$ should be ${}^L G_X={}^LG/Z_{G,H}$ \footnote{this is not true if we do not assume $(Res_{E/F}G_0,Res_{E/F}H_0)$ is strongly tempered for any finite field extension $E$ of $F$}. According to the work of Sakallaridis and Wang (\cite{Sa}, \cite{SW}), there is a representation $\rho_X:{}^LG_X\rightarrow \GL(V)$ of ${}^LG_X$ associated to $(G, H,\xi)$ so that the square of the global period integral associated to $X$ should be related to the central value of the automorphic L-function associated to $\rho_X$. To continue our discussion, we assume the following conjecture (see a similar assumption in Section 5 of \cite{BSV}).

\begin{conj}\label{conj anomaly endoscopy}
The representation $\rho_X$ is symplectic and anomaly free under endoscopy.
\end{conj}

Let $\phi':W_F'\rightarrow {}^LG_X$ be a tempered Langlands parameter. We are going to define a function $\omega_{\phi',\rho_X}$ on $Z_{\phi'}$. For $s\in Z_{\phi'}$, there exists an extended endoscopic triple $(G',s,{}^L\eta)$ of $G$ (not necessarily unique) such that $\phi'$ factors through ${}^L\eta$ (i.e. there exists $\phi_0:W_F'\rightarrow {}^LG'$ such that $\phi'={}^L\eta\circ \phi_0$). Let $T'$ be a maximal quasi-split torus of $G'$. Since $\rho_X$ is anomaly free under endoscopy, the symplectic representation $\rho_{X,s,{}^L\eta,-}$ of ${}^LG'$ is anomaly free. Hence we can decompose the representation $\rho_{X,s,{}^L\eta,-}|_{{}^LT'}$ as $\rho\oplus \rho^\vee$ and there exists a character $\chi$ (resp. $\eta$) of ${}^LT'$ (resp. ${}^LG'$) such that 
$$\det(\rho)=\chi^2\cdot \eta|_{{}^LT}.$$
We define
$$\omega_{\phi',\rho_X}(s)=\eta\circ\phi_0(-1)\epsilon(\frac{1}{2},\rho_{X,s,{}^L\eta,-}\circ \phi_0)\in \{\pm 1\}.$$
It is clear that this definition is independent of the choice of the decomposition $\rho_{X,s,{}^L\eta,-}|_{{}^LT'}=\rho\oplus \rho^\vee$ and $\det(\rho)=\chi^2\eta|_{{}^LT}$. However, it still depends on the choice of the extended endoscopic triple $(G',s,{}^L\eta)$  and the lifting $\phi_0$. To continue our discussion, we assume the following conjecture.

\begin{conj}\label{main conjecture for rho}
The function $\omega_{\phi',\rho_X}$ is well defined (i.e. it is independent of the choice of the extended endoscopic triple and the lifting), it induces a function of $S_{\phi'}$ (i.e. it is constant on each connected component of $Z_{\phi'}$), and it is a character of $S_{\phi'}$.
\end{conj}

Let $\phi:W_F'\rightarrow {}^LG$ be a tempered Langlands parameter of $G(F)$. We would like to define a set of irreducible representations of the component group $S_\phi$. 
Let $I'$ be the set of liftings of $\phi$ to ${}^LG_X$. For each lifting $\phi':W_F'\rightarrow {}^LG_X$ of $\phi$, the above discussion gives us a quadratic character $\omega_{\phi',\rho_X}$ of $S_{\phi'}$ and we also have a map $i$ from $S_{\phi'}$ to $S_\phi$ \footnote{This map is not necessarily injective/surjective}. We use $i(S_{\phi'})$ to denote the image of the map $i$. And we let $I$ be the subset of $I'$ containing those $\phi'$ such that the character $\omega_{\phi',\rho_X}$ is trivial on $ker(i)$.

\begin{defn}
Let $\chi_{\phi',\rho_X,i}$ ($1\leq i\leq |S_\phi/i(S_{\phi'})|$) be the irreducible components of $Ind_{i(S_{\phi'})}^{S_\phi}(\omega_{\phi',\rho_X})$, i.e. $Ind_{i(S_{\phi'})}^{S_\phi}(\omega_{\phi',\rho_X})=\oplus_i \chi_{\phi',\rho_X,i}$. We define
$$I(\phi,\rho_X)=\{\chi_{\phi',\rho_X,i}|\;1\leq i\leq |S_\phi/i(S_{\phi'})|,\;\phi'\in I\}.$$
This is a multi-set, some irreducible representations may appear more than once.
\end{defn}

Now we can formulate our conjecture for the multiplicity. We assume that the map $H^1(F, H)\rightarrow H^1(F, G)$ is injective \footnote{If we do not make this conjecture, then our conjecture would be for the multiplicity of $G/H(F)$, instead of $G(F)/H(F)$. We refer the reader to Section 4 for an example of this kind.} (i.e. $G/H(F)=G(F)/H(F)$). Let $G_{qs}$ be the quasi-split pure inner form of $G$. The Whittaker datum of $G_{qs}(F)$ is a $ker(H^1(F,Z_{G})\rightarrow H^1(F,G))$-torsor.

\begin{conj}\label{main conj}
Let $\pi$ be an irreducible tempered representation of $G(F)$ whose central character is trivial on $Z_{G, H}(F)$, and let $\phi$ be the Langlands parameter of $\pi$. There exists a choice of Whittaker datum of $G_{qs}$ (only depends on $(G, H,\xi)$, in particular, independent of $\pi$) such that under this choice of Whittaker datum, the multiplicity $m(\pi)$ is equal to the number of irreducible representations in $I(\phi,\rho_X)$ that is equal to $\omega_\pi$. Here $\omega_\pi$ is the irreducible representation of $S_\phi$ associated to $\pi$ under the local Langlands correspondence (with respect to the choice of Whittaker datum).

Moreover, the choice of the Whittaker datum is not necessarily unique. All the possible choices form a $Im(ker(H^1(F,Z_{G,H})\rightarrow H^1(F,H))\rightarrow ker(H^1(F,Z_{G})\rightarrow H^1(F,G)))$-torsor.
\end{conj}

\begin{rmk}
The above conjecture is similar to the epsilon dichotomy conjecture for the Gan-Gross-Prasad models in \cite{GGP} and for 10 strongly tempered models in our previous paper \cite{WZ}. But there are two important improvements (both in the definition of $\omega_{\phi',\rho_X}$). First, in \cite{WZ}, when we define $\omega_{\phi',\rho_X}$, we only consider $s\in Z_\phi$ that belongs to an elliptic extended endoscopic triple (this is true for the models in \cite{WZ} but is not true for the general case). In our definition in this paper, we do not require the elliptic condition.

Secondly, in our definition in \cite{WZ}, we explicitly write down $s$ and  define the term $\eta\circ\phi_0(-1)$ in the definition of $\omega_{\phi',\rho_X}$ by an explicit formula. In \cite{GGP}, they explicitly write down a representative for each element of $S_\phi$ and then define the function $\omega_{\phi',\rho_X}$ by an explicit formula. In this paper, we define the term $\eta\circ\phi_0(-1)$ in a conceptual way using the anomaly free property. This is a very important improvement because for general groups (e.g. $E_7, E_8$), it is very hard (at least for us) to explicitly write down the component group $S_\phi$ and representative of elements in $S_\phi$ for the general Langlands parameter.

Another important point in Conjecture \ref{main conj} is that we can consider the case when there is more than one open Borel orbit (i.e. the multiplicity for the L-packet is not necessarily one). This is the first time such a conjecture has been proposed (other than in some lower-rank cases). The key is to use the set of liftings and to consider the induced representation $Ind_{i(S_{\phi'})}^{S_\phi}(\omega_{\phi',\rho_X})$.
\end{rmk}

\subsection{How to prove Conjecture \ref{main conj} and some open questions}
In this subsection, we discuss some ideas about proving Conjecture \ref{main conj} and some open questions. The first step is to prove a multiplicity formula $m(\pi)=m_{geom}(\pi)$ for all tempered representations. Here $m_{geom}(\pi)$ is defined in \cite{Wan} and is called the geometric multiplicity. Such a multiplicity formula has been proved for many strongly tempered spherical varieties such as the Gan-Gross-Prasad models and the models in \cite{WZ}. Moreover, for each given model, it seems that the current trace formula method (invented by Waldspurger in his proof of the orthogonal Gan-Gross-Prasad conjecture \cite{Wal1}, \cite{Wal2}) can be used to prove the multiplicity formula. But it is still not clear at this moment how to write down the proof for the general case without using any feature pertaining to the specific model.

After proving the multiplicity formula, one can study the behavior of the geometric multiplicity under endoscopic. Together with some inductional hypothesis (i.e. we assume the epsilon dichotomy conjecture holds for some models related to the endoscopic group of $G$), we can reduce the proof of Conjecture \ref{main conj} to the computation of the sum of the multiplicity over the L-packet. This idea was invented by Waldspurger in his proof of the orthogonal Gan-Gross-Prasad conjecture \cite{Wal}. As in the proof of the multiplicity formula, it seems that Waldspurger's method can be used for any given model, but it is not clear how to write it for a general case. In particular, if $G'$ is an endoscopic group of $G$, it is not clear in general which models of $G'$ should be related to $(G, H)$. For a specific model, we know the model associated to $G'$ by direct computation, but we do not have a general theory to explain this (i.e. we need a relative endoscopic theory for strongly tempered spherical varieties).

The last step, which is also the most difficult step, is to study the multiplicity of the L-packet. The goal is to relate it to the epsilon factor (under the language of \cite{WZ}, we call this the weak epsilon dichotomy conjecture, or just the weak conjecture). For this step, we do not have a systematic way to solve it at this moment. For the Gan-Gross-Prasad model, this was done by relating the multiplicity of the L-packet to the twisted multiplicity of the Gan-Gross-Prasad model of the general linear group. But this method does not work if the Langlands functoriality $\rho_X:{}^L G\rightarrow \GL(V)$ is not of twisted endoscopic type (in particular it does not apply to any of the cases in \cite{WZ}). For all the models in \cite{WZ} except the model $(\GSp_6\times \GSp_4, G(\Sp_4\times \Sp_2))$, in our recent paper \cite{WZ2}, we proposed a method to prove the weak conjecture using the ``dichotomy" behavior of certain degenerate principal series of $\GSp_6$. The reason this method works is due to the fact that for all the models in \cite{WZ} except the model $(\GSp_6\times \GSp_4, G(\Sp_4\times \Sp_2))$, the epsilon factor can be defined using some local Rankin-Selberg integral involving the degenerate principal series of $\GSp_6$ (in particular this method cannot be used to prove the weak conjecture of the Gan-Gross-Prasad model) \footnote{The reason we exclude the model $(\GSp_6\times \GSp_4, G(\Sp_4\times \Sp_2))$ is that at this moment there is no Rankin-Selberg integral defining the epsilon factor associated to this model.}. It is not clear at this moment how to prove the weak conjecture for the general case (although for all the strongly tempered models we know except the model $(\GSp_6\times \GSp_4, G(\Sp_4\times \Sp_2))$, one of the two methods discussed here can be used to prove the weak conjecture) \footnote{an interesting point is that it seems these two methods are disjoint, we do not know any example where both methods can be used to prove the epsilon dichotomy conjecture.}.

Another open question is regarding the case when $G/H(F)\neq G(F)/H(F)$. In this case, Conjecture \ref{main conj} studies the multiplicity of $G/H(F)$, not the multiplicity of $G(F)/H(F)$. This is compatible with the philosophy of Sakellaridis-Venkatesh in \cite{SV}, but it would be nice to have a conjecture for the multiplicity of $G(F)/H(F)$.

The last open question we will discuss here is about the choice of Whittaker datum in Conjecture \ref{main conj}. In Conjecture \ref{main conj}, we were not able to specify the choice of Whittaker datum, we only conjectured that all the possible choices form an $Im(ker(H^1(F, Z_{G, H})\rightarrow H^1(F, H))\rightarrow ker(H^1(F, Z_{G})\rightarrow H^1(F, G)))$-torsor. Among all the known cases, the Whittaker datum is unique for all the models in \cite{WZ} so this is not an issue; for the Gan-Gross-Prasad model, the Whittaker datum is not unique, and in Section 12 of \cite{GGP} they gave a specific choice of the Whittaker model. But at this moment we do not know how to generalize it to general strongly tempered spherical varieties.

\subsection{Why do we need anomaly free?}\label{sec anomaly free}
In this subsection, we will explain why we need the condition of anomaly free from the point of view of our paper and the work of Ben-Zvi--Sakellaridis--Venkatesh \cite{BSV}. 

From our point of view, the anomaly free condition is used to define the term $\eta\circ\phi_0(-1)$ in the character
$$\omega_{\phi',\rho_X}(s)=\eta\circ\phi_0(-1)\epsilon(\frac{1}{2},\rho_{X,s,{}^L\eta,-}\circ \phi_0)\in \{\pm 1\}.$$
In all the previous epsilon dichotomy conjectures (\cite{GGP}, \cite{WZ}), this term was defined by an explicit computation. It was given so that for most unramified parameter $\phi_0$, the value of $\omega_{\phi',\rho_X}(s)$ should be equal to 1 (this is because for most unramified parameters the component group is trivial and hence we need the character to also be trivial). With the assumption of anomaly free, we know that for most unramified parameter $\phi_0$, the epsilon factor $\epsilon(\frac{1}{2},\rho_{X,s,{}^L\eta,-}\circ \phi_0)$ is equal to $\eta\circ\phi_0(-1)$ and hence we can define the character $\omega_{\phi',\rho_X}$ in this way.

From the point of view of \cite{BSV}, one of the goals in \cite{BSV} is to equip each Hamiltonian $G$-space with the automorphic quantization, without passing the metaplectic cover of $G$. 
They introduce the notion of ``anomaly-free'' to Hamiltonian $G$-spaces in \cite[Definition 5.1.2]{BSV} and conjecture such symplectic varieties admit an automorphic and spectral quantization. 
When $M$ is a symplectic vector space, Definition \ref{defn:anomaly} is equivalent to their definition (see \cite[Proposition 5.1.5]{BSV}). 
Moreover, Examples 5.1.7 and 5.1.9 in \cite{BSV} give more hyperspecial vector spaces examples and elaborate more detailed connections with our table in \cite{WZ} and our example in Remark \ref{rmk:anomaly-cases}.

\section{Know examples with multiplicity one}\label{sec multiplicity one case}
In this section, we will show that our conjecture recovers the epsilon dichotomy conjecture in \cite{GGP} for the Gan--Gross--Prasad model and the epsilon dichotomy conjecture in \cite{WZ} for 10 strongly tempered models.

\subsection{The Whittaker model}
Let $G$ be a quasi-split reductive group defined over $F$, $N$ be a maximal unipotent subgroup of $G$, and $\xi$ be a generic character of $N(F)$. In this case ${}^LG_X={}^LG$ and the representation $\rho_X$ is zero-dimensional.

In this case, it is clear that Conjecture \ref{conj anomaly endoscopy} and \ref{main conjecture for rho}  are satisfied. Moreover, the set $I(\phi,\rho_X)$ contains a unique element which is the trivial character of $S_\phi$. Then Conjecture \ref{main conj} follows from the local Langlands conjecture (in this case, the choice of Whittaker datum is unique and should be the one associated to $\xi$).

\subsection{The Gan-Gross-Prasad model}
In this subsection, we will show that for the Gan--Gross--Prasad models, Conjecture \ref{main conj} is the same as the epsilon dichotomy conjecture in \cite{GGP}. We will only consider the orthogonal group case $(G, H)=(\SO_{a+2b+1}\times \SO_{a},\SO_{a}\ltimes N)$, the unitary group case follows from a similar argument. In this case, ${}^LG_X={}^LG=\Sp_{2m}(\BC)\times \SO_{2n}(\BC)$ or ${}^LG=\Sp_{2m}(\BC)\times O_{2n}(\BC)$ where $\{2m+1,2n\}=\{a+2b+1,a\}$. And the representation $\rho_X$ is the $4mn$-dimensional tensor product representation of ${}^LG_X$. Moreover, $Z_{G, H}$ is trivial and the choice of Whittaker datum is unique (defined in Section 12 of \cite{GGP}).

Let $\phi: W_F'\rightarrow {}^LG$ be a tempered L-parameter. Let $M$ (resp. $N$) be the self-dual representation of $W_F'$ by composing $\phi$ with the standard representation of $\Sp_{2m}(\BC)$ (resp. $O_{2n}(\BC)$ or $\SO_{2n}(\BC)$). As in Section 4 of \cite{GGP}, we can decompose $M$ and $N$ as
\begin{align*}
 M=&\oplus_{i=1}^{a_1} m_{1i}M_{1i}+\oplus_{i=1}^{a_2} 2m_{2i} M_{2i}+\oplus_{i=1}^{a_3} m_{3i}(M_{3i}\oplus M_{3i}^{\vee})\\
 N=&\oplus_{j=1}^{b_1} n_{1j}N_{1j}+\oplus_{j=1}^{b_2} 2n_{2j} N_{2j}+\oplus_{j=1}^{b_3} n_{3j}(N_{3j}\oplus N_{3j}^{\vee}),
\end{align*} 
where $M_{1i},N_{2j}$ are of symplectic type, $M_{2i},N_{1j}$ are of orthogonal type, and $M_{3i},N_{3j}$ are not self-dual. Then $Z_\phi$ and $S_\phi$ are given by
\begin{align*}
Z_\phi=&\Pi_{i=1}^{a_1} O(m_{1i},\BC)\times \Pi_{i=1}^{a_2} \Sp(2m_{2i},\BC)\times \Pi_{i=1}^{a_3} \GL(m_{3i},\BC)\\ 
&\times \Pi_{j=1}^{b_1} O(n_{1j},\BC)\times \Pi_{j=1}^{b_2} \Sp(2n_{2j},\BC)\times \Pi_{j=1}^{b_3} \GL(n_{3j},\BC)\\ 
S_\phi=&(\BZ/2\BZ)^{a_1}\times (\BZ/2\BZ)^{b_1}.
\end{align*}
We just need to show that the function $\omega_{\phi,\rho_X}$ defined in the previous section is the same as the character $\chi_N\times \chi_M$ defined in Section 6 of \cite{GGP}.

We first recall the definition of $\chi_N\times \chi_M$. For $a_M\in (\BZ/2\BZ)^{a_1}$, let $M^{a_M}=\oplus_{i} M_{1i}$ where $i$ runs over all the components of $a_M$ with $-1$ coordinate. Similarly, we can also define $N^{a_N}$ for $a_N\in (\BZ/2\BZ)^{b_1}$. In Section 6 of \cite{GGP}, they define 
\begin{align*}
\chi_N(a_M)\chi_M(a_N)=&\epsilon(M^{a_M}\otimes N)\epsilon(M\otimes N^{a_N})\det(M^{a_M})(-1)^{\dim(N)/2}\\ 
&\times\det(N)(-1)^{\dim(M^{a_M})/2}\det(N^{a_N})(-1)^{\dim(M)/2}\det(M)(-1)^{\dim(N^{a_N})/2}. 
\end{align*}
Here to simplify the notation, for a symplectic representation $V$ of $W_F'$, we use $\epsilon(V)$ to denote $\epsilon(\frac{1}{2},V)$.

Next, we show that $\omega_{\phi,\rho_X}$ coincides with $\chi_N\times \chi_M$. Let 
$$s=(g_{1i},g_{2i},g_{3i},h_{1j},h_{2j},h_{3j})$$
be an element in $Z_\phi$ with
\begin{align*}
 &g_{1i}\in O(m_{1i},\BC),\;g_{2i}\in \Sp(2m_{2i},\BC),\;g_{3i}\in \GL(m_{3i},\BC),\\ 
 &h_{1j}\in O(n_{1j},\BC),\;h_{2j}\in \Sp(2n_{2j},\BC),\;h_{3j}\in \GL(n_{3j},\BC). 
\end{align*} 
Let $a_M\times a_N$ be the corresponding element in $S_\phi$. We let $I_1$ (resp. $J_1$) be the set of $1\leq i\leq a_1$ (resp. $1\leq j\leq b_1$) such that $g_{1i}\in O(m_{1i},\BC)-\SO(m_{1i},\BC)$ (resp. $h_{1i}\in O(n_{1j},\BC)-\SO(n_{1j},\BC)$) and let $I_2$ (resp. $J_2$) be the complement of $I_1$ (resp. $J_1$) in $\{1,2,\cdots,a_1\}$ (resp. $\{1,2,\cdots, b_1\}$). We let $I_{1,odd}$ (resp. $I_{1,even}$) be the set of $i\in I_1$ such that $m_{1i}$ is odd (resp. $even$). Similarly, we can define $I_{2,odd}, I_{2,even},J_{1,odd}.J_{1,even},J_{2,odd},J_{2,even}$. By Proposition 5.1 of \cite{GGP}, we have
\begin{align*}
\chi_N\times \chi_M(s)=&\Pi_{i\in I_1}\Pi_{1\leq j\leq b_1} \epsilon(M_{1i}\otimes N_{1j})^{n_{1j}} \det(N_{1j})^{\frac{n_{1j}\cdot \dim(M_{1j})}{2}}\\
&\times \Pi_{j\in J_1} \Pi_{1\leq i\leq a_1}\epsilon(M_{1i}\otimes N_{1j})^{m_{1i}}\det(N_{1j})^{\frac{m_{1i}\cdot \dim(M_{1j})}{2}}\\ 
=&\Pi_{(i,j)\in I_1\times J_{2,odd}\cup I_{2,odd}\times J_1\cup I_{1,even}\times J_{1,odd}\cup I_{1,odd}\times J_{1,even}} \epsilon(M_{1i}\otimes N_{1j})\det(N_{1j})^{\frac{\dim(M_{1j})}{2}}.
\end{align*}

Next, we study the $-1$-eigenspace $V_{s,-}$ of $\rho_X(s)$. The eigenspace $V_{s,-}$ is a direct sum of $-1$-eigenspace associated to $g_{ki}\times h_{lj}$ with $1\leq k,l\leq 3$. We will study them separately. 

We first study the $-1$-eigenspace associated to $g_{1i}\times h_{2j}$. By using the tensor representation we can view $g_{1i}\times h_{2j}$ as an element in $\GL(m_{1i}n_{2j})$ and we let $2k$ be the dimension of the $-1$-eigenspace of this matrix (it is easy to see that this dimension is an even number). Then it is easy to see that the $-1$-eigenspace associated to $g_{1i}\times h_{2j}$ is $2k$-copy of $M_{1i}\otimes N_{2j}$. This representation is obviously anomaly free and we can choose the character $\eta$ in the definition of anomaly free to be trivial for this representation. Moreover, by Proposition 5.1 of \cite{GGP}, the epsilon factor associated to it is also equal to 1. Hence the contribution of this $-1$-eigenspace to the character $\omega_{\phi,\rho_X}$ is just 1.

Similarly, we can show that the contribution of the $-1$-eigenspaces coming from 
$$g_{1i}\times h_{3j},\;g_{2i}\times h_{1j},\;g_{2i}\times h_{2j},\;g_{2i}\times h_{3j},\;g_{3i}\times h_{1j},\;g_{3i}\times h_{2j},\;g_{3i}\times h_{3j}$$
to the character $\omega_{\phi,\rho_X}$ is also 1.

It remains to consider the $-1$-eigenspace associated to $g_{1i}\times h_{1j}$. By a similar argument as above we can show that the contribution of the $-1$-eigenspaces coming from  $g_{1i}\times h_{1j}$ with
$$(i,j)\in I_1\times J_{2,even}\cup I_{2,even}\times J_1\cup I_{1,odd}\times J_{1,odd}\cup I_{1,even}\times J_{1,even}\cup I_2\times J_2$$
to the character $\omega_{\phi,\rho_X}$ is also 1.

Next, we consider the $-1$-eigenspace associated to $g_{1i}\times h_{1j}$ with $(i,j)\in I_1\times J_{2,odd}$. By using the tensor representation we can view $g_{1i}\times h_{1j}$ as an element in $\GL(m_{1i}n_{2j})$ and we let $k$ be the dimension of the $-1$-eigenspace of this matrix (it is easy to see that this dimension is an odd number). Then it is easy to see that the $-1$-eigenspace associated to $g_{1i}\times h_{1j}$ is $k$-copy of $M_{1i}\otimes N_{1j}$. This representation is obviously anomaly free and we can choose the character $\eta$ in the definition of anomaly free to be $\det(N_{1j})^{\frac{\dim(M_{1j})}{2}}$. Moreover, the epsilon factor associated to it is equal to $\epsilon(M_{1i}\otimes N_{1j})^k=\epsilon(M_{1i}\otimes N_{1j})$. Hence the contribution of this $-1$-eigenspace to $\omega_{\phi,\rho_X}$ is $\epsilon(M_{1i}\otimes N_{1j})\det(N_{1j})^{\frac{\dim(M_{1j})}{2}}$.

Similarly, we can show that the contribution of the $-1$-eigenspace associated to 
$$g_{1i}\times h_{1j},\;(i,j)\in I_{2,odd}\times J_1\cup I_{1,even}\times J_{1,odd}\cup I_{1,odd}\times J_{1,even}$$
to $\omega_{\phi,\rho_X}$ is also $\epsilon(M_{1i}\otimes N_{1j})\det(N_{1j})^{\frac{\dim(M_{1j})}{2}}$. This implies that $\omega_{\phi,\rho_X}$ is the same as $\chi_N\times \chi_M$. In particular, we have proved that for the Gan--Gross--Prasad model, Conjecture \ref{main conj} is the same as the epsilon dichotomy conjecture in \cite{GGP}.

\subsection{The models in \cite{WZ}}
In this subsection, we will show that Conjecture \ref{main conj} recovers the epsilon dichotomy conjecture of the 10 models considered in \cite{WZ}. We will only consider the most complicated model $(E_7,\PGL_2\ltimes N)$. The other models in \cite{WZ} follows from a similar and easier argument. We will first prove Conjecture \ref{conj anomaly endoscopy}. Then as in \cite{WZ}, by assuming the weak conjecture (Conjecture 1.6 of \cite{WZ}) holds for the model $(E_7,\PGL_2\ltimes N)$, we will prove Conjecture \ref{main conj}.

We first prove Conjecture \ref{conj anomaly endoscopy} for this model. In this case, ${}^LG_X={}^LG=\hat{G}\times W_F'$ where $\hat{G}=E_{7,sc}(\BC)$ is the simply connected form of $E_7$ and $\rho_X$ is the 56-dimensional representation of $E_{7,sc}(\BC)$. To prove Conjecture \ref{conj anomaly endoscopy}, we only need to prove the following proposition.

\begin{prop}
For $s\in \hat{G}_{ss}$, let $V_{s,-}$ be the $-1$-eigenspace of $\rho_X(s)$. The representation of $\hat{G}_{s}$ on $V_{s,-}$ is anomaly free.
\end{prop}

\begin{proof}
We use $\rho_{X,s,-}$ to denote the representation of $\hat{G}_{s}$ on $V_{s,-}$. If $s$ is elliptic, the representation $\rho_{X,s,-}$ was described in Section 2.5 of our previous paper \cite{WZ}. From there it is easy to see that $\rho_{X,s,-}$ is anomaly free, we just need to use the fact that the following representations are anomaly free (which follows from an easy direct computation):
\begin{itemize}
\item the representation $\rho_X$ of $\hat{G}$;
\item the tensor representation of $Spin_{12}(\BC)\times \SL_2(\BC)/(\BZ/2\BZ)$;
\item the Half-Spin representation of $Spin_{12}(\BC)$;
\item the exterior square L-function of $\SL_6(\BC)/(\BZ/3\BZ)$;
\item the representation $\wedge^2\otimes std$ representation of $\SL_4(\BC)\times \SL_2(\BC)/(\BZ/4\BZ)$.
\end{itemize}

When $s$ is not elliptic, consider the following Dynkin diagram of $E_7$.

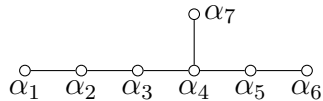
\begin{figure}[h!]
\begin{tikzpicture}[inner sep=0.5mm,scale=0.75]
\node [circle,draw,label=below:${\alpha_1}$] (1)  at ( 0,0)  {}; 
\node [circle,draw,label=below:${\alpha_2}$] (3) at ( 1,0) {};
\node [circle,draw,label=below:${\alpha_3}$] (4) at ( 2,0) {};
\node [circle,draw,label=below:${\alpha_4}$] (5) at ( 3,0) {}; 
\node [circle,draw,label=below:${\alpha_5}$] (6)  at (4,0) {};
\node [circle,draw,label=below:${\alpha_6}$] (7) at ( 5,0) {};
\node [circle,draw,label=right:${\alpha_7}$] (2)  at ( 3,1) {}; 
\draw  (1) -- (3);
\draw  (3) -- (4);
\draw  (2) -- (5);
\draw  (4) -- (5);
\draw  (5) -- (6);
\draw  (6) -- (7);
\end{tikzpicture}

\end{figure}
Let $L_i$ be the maximal Levi subgroup of $\hat{G}$ associated to the simple roots $\Delta-\{\alpha_i\}$ for $1\leq i\leq 7$ with $\Delta=\{\alpha_1,\alpha_2,\cdots,\alpha_7\}$. It is enough to show that $\rho_X|_{L_i}$ is anomaly free under endoscopy for all $i$. If $i=4,6,7$, the restriction of $\rho_X$ to $L_i$ is of the form $\rho_X=\rho\oplus \rho^\vee$ for some representation $\rho$ of $L_i$. Hence we know that $\rho_X|_{L_i}$ is anomaly free under endoscopy. 

If $i=1$ (resp. 2, 3, 5), then the restriction $\rho_X$ to $L_i$ is of the form $\rho\oplus \rho^\vee\oplus \rho'$ where $\rho$ is a representation of $L_i$ and $\rho'$ is the Half-Spin representation of $Spin_{12}(\BC)\times \GL_1(\BC)/(\BZ/2Z)$ (resp. the exterior cube representation of $\SL_6(\BC)/(\BZ/3\BZ)$, $\wedge^2\otimes std$ representation of $\SL_4(\BC)\times \SL_2(\BC)/(\BZ/4\BZ)$, the tensor product representation of $Spin_{10}(\BC)\times \SL_2(\BC)/(\BZ/4\BZ)$). It remains to show that $\rho'$ is anomaly free under endoscopy. The argument is the same as above and we will skip it here (i.e. we first consider the elliptic elements of $L_i$ for which we can explicitly write down the representation $V_{s,-}$ and show that it is anomaly free, then we can further reduce to maximal Levi subgroup of $L_i$). This proves the proposition.
\end{proof}

Next, we assume the weak conjecture (Conjecture 1.6 of \cite{WZ}) holds for the model $(E_7,\PGL_2\ltimes N)$, we will prove Conjecture \ref{main conj}. Let $(G,H)$ and $(G_D,H_D)$ be as in Section 8 of \cite{WZ}. Let $\phi:W_F'\rightarrow {}^LG$ be a tempered L-parameter of $G$ and let $\omega_{\phi}$ be the character of $S_\phi$ corresponds to the unique distinguished element in the L-packet $\Pi_\phi=\Pi_\phi(G)\cup \Pi_\phi(G_D)$. We need to show that $\omega_{\phi,\rho_X}=\omega_{\phi}$. In our previous paper (Section 2 of \cite{WZ}), we have defined a function $\omega_{\phi, H}$ on the elliptic elements of $Z_\phi$ and we have proved in Section 8 of \cite{WZ} that $\omega_{\phi, H}=\omega_\phi$ on the elliptic elements of $Z_\phi$. It is clear from the definition that the functions $\omega_{\phi, H}$ and $\omega_{\phi,\rho_X}$ are the same on all the elliptic elements formulas. Hence by using the result in Section 8 of \cite{WZ} we proved that  $\omega_{\phi,\rho_X}(s)=\omega_{\phi, H}(s)$ for all $s\in Z_\phi$ with $s$ elliptic. 

When $s$ is not elliptic, the argument is the same as the elliptic case in Section 8 of \cite{WZ}. Namely, let $(G',s,{}^L\eta)$ be the extended endoscopic triple such that $\phi$ factors through ${}^L\eta$ (i.e. there exists $\phi_0:W_F'\rightarrow {}^LG'$ such that $\phi={}^L\eta\circ \phi_0$). Then we study the behavior of the geometric multiplicity of the model under endoscopy between $G$ and $G'$. When $s$ is elliptic, this is done in our previous paper (Section 8.3 of \cite{WZ}). If $s$ is non-elliptic, the argument is the same. We can first pass from $G$ to its Levi subgroup $L_s$ (this step has already been done in Proposition 8.1 of \cite{WZ}), then we can study the endoscopic transfer of the geometric multiplicity between $L_s$ and $G'$ (the argument is the same as the one in Section 8.3 \cite{WZ}). After we proved this identity, we get some formulas of distributions on $G'$ which can be related to some models of $G'$. Then by using the weak conjecture, we can relate it to a certain epsilon factor and prove that $\omega_{\phi,\rho_X}(s)=\omega_{\phi, H}(s)$. Since the argument is the same as the elliptic case in Section 8 of \cite{WZ}, we will only list the models related to $G'$ and skip the remaining details.

\begin{itemize}
\item If $L_s$ is associated to one of the following subsets of $\Delta$ 
$$\{\alpha_4,\alpha_6,\alpha_7\},\{\alpha_4,\alpha_6,\alpha_7,\alpha_i\},\;i=1,2,3,5,$$
$$\{\alpha_4,\alpha_6,\alpha_7,\alpha_i,\alpha_j\},\;\{i,j\}=\{1,2\},\{1,3\},\{1,5\},\{2,5\},\{3,5\},$$
$$\Delta\smallsetminus\{\alpha_i\},\;i=2,3,$$
it is of Type A and we must have $G'=L_s$. In this case, the model related to $G'$ is the model $(G', G'\cap H)$ as in Proposition 8.1 of \cite{WZ}.
\item If $L_s$ is associated to $\{\alpha_2,\alpha_3,\alpha_4,\alpha_6,\alpha_7\}$, $L_s$ is of Type $D_4\times A_1$. In this case, $G'$ is either equal to $L_s$ or of the type $(A_1)^5$. If $G'$ is of Type $(A_1)^5$, the model related to $G'$ is the trilinear $\GL_2$ model. If $G'=L_s$, the model related to $G'$ is the model $(G',G'\cap H)$.
\item If $L_s=L_5$, $L_s$ is of Type $D_5\times A_1$. In this case, $G'$ is either equal to $L_s$ or of the type $A_3\times A_1\times A_1\times A_1$. If $G'=L_s$, the model related to $G'$ is the model $(G',G'\cap H)$. If $G'$ is of the Type $A_3\times A_1\times A_1\times A_1$, the models related to $G'$ are the model $(\GL_4\times \GL_2,\GL_2\times \GL_2)$ and the trilinear $\GL_2$-model.
\item If $L_s=L_1$, $L_s$ is of Type $D_6$. In this case, $G'$ is equal to $L_s$, of Type $D_4\times A_1\times A_1$ or of Type $A_3\times A_3$. If $G'=L_s$, the model related to $G'$ is the model $(G',G'\cap H)$. If $G'$ is of the Type $D_4\times A_1\times A_1$, the model related to $G'$ is the model $(\GSO_8\times \GL_2,\GL_2\ltimes N)$. If $G'$ is of Type $A_3\times A_3$, the model related to $G'$ is the Whittaker model.
\item For all the other cases, the model related to $G'$ is the Whittaker model.
\end{itemize}

\section{Some other examples without multiplicity one}\label{sec non-multiplicity one case}
In this section, we discuss some models with more than one Borel orbits. We will show that our conjecture holds for these cases.

\subsection{The model $(\SL_2,\GL_1)$}

In this section we consider the model $(G,H)=(\SL_2,\GL_1)$ (i.e. $H$ is a maximal split tori of $G$). In this case, ${}^LG=\PGL_2(\BC)$ and ${}^LG_X=\SL_2(\BC)$. For a tempered parameter $\phi$ of $G(F)$, the central character of the packet $\Pi_\phi(G)$ is trivial if and only if there is a lifting $\phi'$ of $\phi$ to $G_X(F)$. In this case, the set of liftings $I'$ contains $|F^{\times}/(F^{\times})^2|/|S_\phi|$ many elements (any two different liftings are differed by a twist of quadratic character). The L-packet $\Pi_\phi(G)$ contains $|S_\phi|$ many representations and it is easy to see that each of them has multiplicity $|F^{\times}/(F^{\times})^2|/|S_\phi|=|I'|$. In fact, let $J$ be the set of quadratic characters $\eta$ such that $\phi'\simeq \phi'\otimes \eta$. Then $|I'|=|F^{\times}/(F^{\times})^2|/|J|$ and $|S_\phi|=|J|$.

The representation $\rho_X$ of ${}^LG=\SL_2(\BC)$ is just $std\oplus std$. It is easy to see that Conjecture \ref{conj anomaly endoscopy}, \ref{main conjecture for rho} hold and the character $\omega_{\phi',\rho_X}$ of $S_{\phi'}$ is just the trivial character. As a result, the set $I$ is equal to $I'$ and the set 
$$\{\chi_{\phi',\rho_X,i}|\;1\leq i\leq |S_\phi/S_{\phi'}|,\;\phi'\in I\}$$
contains all the characters of $S_\phi$, each of them appears exactly $|I|=|F^{\times}/(F^{\times})^2|/|S_\phi|$ times. The choice of Whittaker datum does not matter in this case since the map $\ker(H^1(F, Z_{G, H})\rightarrow H^1(F, H)) \rightarrow \ker(H^1(F, G)\rightarrow H^1(F, G))$ is a bijection. This proves Conjecture \ref{main conj}.

\subsection{The model $(\SL_2,E^1)$}
In this section, we consider the model $(G,H)=(\SL_2,E^1)$ where $E/F$ is a quadratic extension, $\eta_{E/F}$ is the quadratic character associated to $E/F$ and $E^1=ker(\eta_{E/F})$ (i.e. $H$ is a maximal elliptic tori of $F$). As in the previous case, we have ${}^LG=\PGL_2(\BC)$ and ${}^LG_X=\SL_2(\BC)$. For a tempered parameter $\phi$ of $G(F)$, the central character of the packet $\Pi_\phi(G)$ is trivial if and only if there is a lifting $\phi'$ of $\phi$ to $G_X(F)$, the set of liftings $I'$ contains $|F^{\times}/(F^{\times})^2|/|S_\phi|$ many elements and the L-packet $\Pi_\phi(G)$ contains $|S_\phi|$ many representations.

For this model, $X(F)$ is not equal to $G(F)/H(F)$ and it is equal to $X(F)=G(F)/H(F)\cup G(F)/H'(F)$ where $H'(F)$ is another maximal elliptic tori of $G(F)$ that is isomorphic to $E^1$ (if $\eta_{E/F}(-1)=1$ then $H'$ is not conjugated to $H$; if $\eta_{E/F}(-1)=-1$ then we may just choose $H'$ to be $H$). Hence Conjecture \ref{main conj} studies the multiplicity
$$m(\pi)=\dim(\Hom_{H(F)}(\pi,1))+\dim(\Hom_{H'(F)}(\pi,1)).$$
In this case, any maximal elliptic tori of $G(F)$ that is isomorphic to $E^1$ is either conjugated to $H$ or $H'$. Since any two representations in the L-packet $\Pi_\phi(G)$ can be conjugated to each other by an element of $\GL_2(F)$, we know that the multiplicity $m(\pi)$ is constant among representations in the L-packet $\Pi_\phi(G)$. 

If we view $\PGL_2$ (resp. $E^1$) as $\SO_3$ (resp. $\SO_2$), then $\rho_X$ is just the 4-dimensional tensor representation of ${}^L\SO_3\times {}^L\SO_2$. It is easy to see that Conjecture \ref{conj anomaly endoscopy} and Conjecture \ref{main conjecture for rho} hold in this case. Moreover, the component group $S_{\phi'}$ ($\phi'\in I'$ and $\pi_{\phi'}$ is the irreducible tempered representation of $\PGL_2(F)$ associated to $\phi'$) is either trivial or $\BZ/2\BZ$. If it is trivial, we have
$$\eta_{E/F}(-1)\epsilon(\frac{1}{2},\pi_{\phi'},\rho_X)=1.$$
If it is equal to $\BZ/2\BZ$, then the character $\omega_{\phi',\rho_X}$ is trivial (resp. the sign character) if $\eta_{E/F}(-1)\epsilon(\frac{1}{2},\pi_{\phi'},\rho_X)=1$ (resp. $\eta_{E/F}(-1)\epsilon(\frac{1}{2},\pi_{\phi'},\rho_X)=-1$). This implies that
$$I=\{\phi'\in I'|\;\eta_{E/F}(-1)\epsilon(\frac{1}{2},\pi_{\phi'},\rho_X)=1\}$$
and the set 
$$\{\chi_{\phi',\rho_X,i}|\;1\leq i\leq |S_\phi/S_{\phi'}|,\;\phi'\in I\}$$
contains all the characters of $S_\phi$, each of them appears exactly $|I|$ times. Moreover, for any $\phi'\in I'$, if we let $J'$ be the set of quadratic characters $\eta$ such that 
$$\eta_{E/F}(-1)\epsilon(\frac{1}{2},\pi_{\phi'}\otimes \eta,\rho_X)=1,$$
then $|I|=\frac{|J'|}{|J|}$. Here we recall from the previous subsection that $J$ is the set of quadratic characters $\eta$ such that $\phi'\simeq \phi'\otimes \eta$.

Let $\phi'$ be an element in $I'$. We can view $\pi_{\phi'}$ as a tempered representation of $\GL_2(F)$ with trivial central character and we have $\Pi_\phi(G)=\pi_{\phi'}|_{\SL_2(F)}$. The model $(\PGL_2,E^1)$ is the famous Waldspurger model and we let $m'(\pi_{\phi'})$ be the multiplicity of $\pi_{\phi'}$ with respect to this model.  The epsilon dichotomy conjecture for the Waldspurger model implies that 
$$m'(\pi_{\phi'})=1\iff \eta_{E/F}(-1)\epsilon(\frac{1}{2},\pi_{\phi'},\rho_X)=1;\;m'(\pi_{\phi'})=0\iff \eta_{E/F}(-1)\epsilon(\frac{1}{2},\pi_{\phi'},\rho_X)=-1.$$
Hence we have ($\eta$ runs over all the quadratic characters modulo the subgroup $\{1,\eta_{E/F}\}$)
$$\dim(\Hom_{H(F)}(\Pi_\phi(G),1))=\dim(\Hom_{H'(F)}(\Pi_\phi(G),1))=\sum_{\eta} m'(\pi_{\phi'}\otimes \eta)=\frac{|J'|}{2}=\frac{|I|\cdot |S_\phi|}{2}.$$
This implies that for any $\pi\in \Pi_\phi(G)$, we have $m(\pi)=|I|$. This proves Conjecture \ref{main conj}.

\begin{rmk}
By a similar argument we can also verify Conjecture \ref{main conj} for the triple product model of $\SL_2$ (i.e. $(G,H)=((\SL_2)^3,\SL_2)$) and $U_2$ (i.e. $(G,H)=((U_2)^3,U_2)$).
\end{rmk}

\subsection{The model $(U_6,U_2\ltimes N)$}
In this subsection, we discuss the unitary Ginzburg-Rallis model $(U_6, U_2\ltimes N)$ studied in our previous paper \cite{WZ1}. In this case, ${}^LG=\GL_6(\BC)\ltimes \BZ/2\BZ$ and ${}^LG_X=\SL_6(\BC)\ltimes \BZ/2\BZ$. For a tempered parameter $\phi$ of $G(F)$, the central character of the packet $\Pi_\phi(G)$ is trivial if and only if there is a lifting $\phi'$ of $\phi$ to $G_X(F)$. In this case, the map $S_{\phi'}\rightarrow S_\phi$ is injective, and we have $S_\phi=S_{\phi'}$ (resp. $|S_\phi/S_{\phi'}|=2$) if and only if $\phi'$ is not isomorphic to $\phi'\otimes \eta_{E/F}$ (resp. $\phi'$ is isomorphic to $\phi'\otimes \eta_{E/F}$). The set of liftings $I'$ contains $|2\cdot S_{\phi'}/S_{\phi}|$ many elements and we have $I=I'$.

The representation $\rho_X$ of ${}^LG_X=\SL_6(\BC)\ltimes \BZ/2\BZ$ is the 20-dim exterior cube representation. By a similar but easier argument as the Gan--Gross--Prasad model case we can prove Conjecture \ref{conj anomaly endoscopy} and \ref{main conjecture for rho} for this model. In this case, the set
$$\{\chi_{\phi',\rho_X,i}|\;1\leq i\leq |S_\phi/S_{\phi'}|,\;\phi'\in I\}$$
contains 2 elements. If the lifting is unique, then $|S_\phi/S_{\phi'}|=2$ and these two characters are the two characters of $S_\phi$ whose restriction to $S_{\phi'}$ is equal to $\omega_{\phi',\rho_X}$. If there are two liftings, then $S_\phi=S_{\phi'}$ and the two characters are just $\omega_{\phi',\rho_X}$ for $\phi'\in I$ (this is equivalent to we only consider one lifting but we consider $\rho_X$ and $\rho_X\otimes \eta_{E/F}$).

The choice of the Whittaker model does not matter since the map $\ker(H^1(F, Z_{G, H})\rightarrow H^1(F, H)) \rightarrow \ker(H^1(F, G)\rightarrow H^1(F, G))$ is a bijection.

We will prove Conjecture \ref{main conj} in this case by assuming Conjecture 1.6 of \cite{WZ} holds for the model $(GU_6, GU_2\ltimes N)$ \footnote{In our paper \cite{WZ2} we have proposed a method to prove this conjecture, and we will prove it in our next paper.}. Like the argument in Section 5 of \cite{WZ}, the key is to study the behavior of the multiplicity under endoscopy.

First, we consider the case when there are two liftings. In this case, the two choices of Whittaker data give the same parametrization of the L-packet. We use $\phi_i'$ ($i=1,2$) to denote these two liftings. In this case $S_\phi=S_{\phi_i'}$. For $s\in S_\phi=S_{\phi_i'}$, as in Section 5.4 of \cite{WZ}, we can choose $s'\in sZ_{\phi_i'}^{\circ}$ so that $s'$ is conjugated to $\pm I_6$ or $\pm diag(I_4,-I_2)$. The value of $\omega_{\phi_i',\rho_X}(s)$ is defined in Section 2.5 of \cite{WZ}. Let 
$$I_G=\{i|\;\eta_{E/F}(-1)\epsilon(\frac{1}{2},\Pi_{\phi_i'},\rho_X)=\varepsilon_G\}$$
where $\varepsilon_G$ is equal to 1 (resp. $-1$) if $G$ is quasi-split (resp. non quasi-split). Then $I_G$ is the set of $i$ such that $\omega_{\phi_i',\rho_X}$ corresponds to a representation in $\Pi_\phi(G)$.

By Conjecture 1.6 of \cite{WZ} and Proposition 5.2 of \cite{WZ1}, the multiplicity of the L-packet $\Pi_\phi(G)$ is equal to $|I_G|$. If $|I_G|=0$, then the two characters $\omega_{\phi_i',\rho_X}$ do not correspond to a representation in $\Pi_\phi(G)$ (both of them correspond to a representation of the pure inner form of $G$). This proves Conjecture \ref{main conj}. If $|I_G|=2$, then the two characters $\omega_{\phi_i',\rho_X}$ both correspond to a representation in $\Pi_\phi(G)$. Also in this case the L-packet $\Pi_\phi(G)$ has multiplicity two and we let $\omega_{\phi, i}$ ($i=1,2$) be the two characters correspond to the distinguished elements in $\Pi_\phi(G)$ (these two characters may be the same).
In this case, by the same argument as in Section 5.4 of \cite{WZ}, we can show that (note that in this case the multiplicity formula was proved in \cite{WZ1} and we can prove the endoscopic identity of the geometric multiplicity by the same argument as in Proposition 5.8 of \cite{WZ})
$$\omega_{\phi,1}(s)+\omega_{\phi,2}(s)=\sum_{\pi\in \Pi_\phi(G)}\chi_\pi(s)m(\pi)=\omega_{\phi_1',\rho_X}(s)+\omega_{\phi_2',\rho_X}(s),\;\forall s\in S_\phi.$$
This implies that $\{\omega_{\phi,1},\omega_{\phi,2}\}=\{\omega_{\phi_1',\rho_X},\omega_{\phi_2',\rho_X}\}$ and proves Conjecture \ref{main conj} in this case. If $|I_G|=1$, we may assume that $\phi_1'\in I_G$. In this case, the L-packet $\Pi_\phi(G)$ has multiplicity one and we let $\omega_{\phi}$  be the character corresponding to the distinguished elements in $\Pi_\phi(G)$. In this case, by the same argument as in Section 5.4 of \cite{WZ}, we can show that 
$$\omega_{\phi}(s)=\sum_{\pi\in \Pi_\phi(G)}\chi_\pi(s)m(\pi)=\omega_{\phi_1',\rho_X}(s),\;\forall s\in S_\phi.$$
This implies that $\omega_{\phi}=\omega_{\phi_1',\rho_X}$ and proves Conjecture \ref{main conj}.

Next, we consider the case when the lifting is unique and we use $\phi'$ to denote this lifting. In this case, if $\eta_{E/F}(-1)\epsilon(\frac{1}{2},\Pi_{\phi'},\rho_X)=-\varepsilon_G$, the multiplicity of the L-packet $\Pi_\phi(G)$ is equal to 0 and the two characters in $I_{S_{\phi'}}^{S_\phi}(\omega_{\phi',\rho_X})$ does not correspond to a representation in $\Pi_\phi(G)$ (both of them corresponds to a representation of the pure inner form of $G$). This proves Conjecture \ref{main conj}. If $\eta_{E/F}(-1)\epsilon(\frac{1}{2},\Pi_{\phi'},\rho_X)=\varepsilon_G$, the multiplicity of the L-packet $\Pi_\phi(G)$ is equal to 2 and the two characters in $I_{S_{\phi'}}^{S_\phi}(\omega_{\phi',\rho_X})$ both correspond to a representation in $\Pi_\phi(G)$. We let $\chi_1,\chi_2$ be these two characters and we let $\omega_{\phi, i}$ ($i=1,2$) be the two characters corresponding to the distinguished elements in $\Pi_\phi(G)$. In this case, the two representations that correspond to $\{\chi_1,\chi_2\}$ are independent of the choice of the Whittaker datum. In fact, if one choice of Whittaker datum $\omega_{\phi, i}$ corresponds to some representation $\pi_i$ of the L-packet, then under the other choice of Whittaker datum $\omega_{\phi,1}$ (resp. $\omega_{\phi,2}$) corresponds to $\pi_2$ (resp. $\pi_1$).

By the same argument as in Section 5.4 of \cite{WZ}, we can show that 
$$\omega_{\phi,1}(s)+\omega_{\phi,2}(s)=\sum_{\pi\in \Pi_\phi(G)}\chi_\pi(s)m(\pi)=\chi_1(s)+\chi_2(s),\;\forall s\in S_{\phi'}.$$
Since $\chi_1$ and $\chi_2$ are the only two characters of $S_\phi$ whose restriction to $S_{\phi'}$ is equal to $\omega_{\phi',\rho_X}$, we only need to show that $\omega_{\phi,1}\neq \omega_{\phi,2}$. Choose $s\in S_\phi-S_{\phi'}$ and we just need to show that $\omega_{\phi,1}(s)+\omega_{\phi,2}(s)=0$. We can fine $s'\in sZ_{\phi}^{\circ}$ such that $s'$ belongs to an elliptic endoscopic triple $(G',s',{}^L\eta)$ of $G$ with $G'=U_5\times U_1$ or $G'=U_3\times U_3$ and $\phi$ factors through ${}^LG'$. Then we need to study the behavior of the geometric multiplicity under the endoscopic transfer between $G$ and $G'$. By a similar but easier argument as in Proposition 5.8 of \cite{WZ}, we can show that $m_{geom}(\theta)=0$ if $\theta$ is the endoscopic transfer of a stable distribution $\theta'$ of $G'(F)$. Here $\theta$ is a quasi-character on $G(F)$ and the geometric multiplicity $m_{geom}(\theta)$ is defined in Section 5.2 of \cite{WZ1}. This implies that
$$\omega_{\phi,1}(s)+\omega_{\phi,2}(s)=\sum_{\pi\in \Pi_\phi(G)}\chi_\pi(s)m(\pi)=0.$$
Hence $\omega_{\phi,1}\neq \omega_{\phi,2}$ and this finishes the proof of Conjecture \ref{main conj}.

\end{document}